\documentclass[a4paper,11pt]{amsart}

\usepackage{mathrsfs}
\usepackage{bm}
\usepackage{color}
\usepackage{multirow,bigdelim}
\usepackage[all]{xy}
\usepackage{hyperref}
\usepackage{tikz}
\usepackage{geometry}
\newtheorem{thm}{Theorem}[section]
\newtheorem{lem}[thm]{Lemma}

\newtheorem{claim}[thm]{Claim}

\theoremstyle{definition}
\newtheorem{defin}[thm]{Definition}

\newtheorem*{conventions}{Conventions}
\newtheorem*{notation}{Notation}
\newtheorem*{acknowledgment}{Acknowledgment}

\theoremstyle{remark}
\newtheorem{rem}[thm]{Remark}

\numberwithin{equation}{section}

\newcommand{\bC}{\mathbb{C}}

\newcommand{\bP}{\mathbb{P}}
\newcommand{\bQ}{\mathbb{Q}}

\newcommand{\bZ}{\mathbb{Z}}

\newcommand{\Sing}{\mathrm{Sing}}
\newcommand{\Supp}{\mathrm{Supp}}

\begin{document}

\title[Minimal compactifications of $\bC^2$ with only star-shaped singularities]{Minimal compactifications of the affine plane with only star-shaped singularities}

%    Remove any unused author tags.

%    author one information
\author{}
\address{}
\curraddr{}
\email{}
\thanks{}

%    author two information
\author{Masatomo Sawahara}
\address{Faculty of Education, Hirosaki University, Bunkyocho 1, Hirosaki-shi, Aomori 036-8560, JAPAN}
\curraddr{}
\email{sawahara.masatomo@gmail.com}
\thanks{The author is supported by JSPS KAKENHI Grant Number JP24K22823. }

\subjclass[2020]{14J17, 14J26, 14R10, 32J05. }

\keywords{minimal compactification of the affine plane, star-shaped singular point. }

\date{}

\dedicatory{}
%%%%%%%%%%%%%%%%%%%%%%%%%%%%%%%%%%%%%%%%%%%%%%%%%%%%%%%%%%%%%%%%%%%%%%%%%%%%%%%%%%%%%%%%%%%%%%%%%%%%%%%%%%
\begin{abstract}
We consider minimal compactifications of the complex affine plane. 
Minimal compactifications of the affine plane with at most log canonical singularities are classified. 
Moreover, every minimal compactification of the affine plane with at most log canonical singularities has only star-shaped singular points. 
In this article, we classify minimal compactifications of the affine plane with only star-shaped singular points. 
%As a byproduct, we construct many examples of numerically del Pezzo surfaces of rank one with singularities worse than log canonical singularities. 
\end{abstract}
%%%%%%%%%%%%%%%%%%%%%%%%%%%%%%%%%%%%%%%%%%%%%%%%%%%%%%%%%%%%%%%%%%%%%%%%%%%%%%%%%%%%%%%%%%%%%%%%%%%%%%%%%%
\maketitle
\setcounter{tocdepth}{1}
%%%%%%%%%%%%%%%%%%%%%%%%%%%%%%%%%%%%%%%%%%%%%%%%%%%%%%%%%%%%%%%%%%%%%%%%%%%%%%%%%%%%%%%%%%%%%%%%%%%%%%%%%%
Throughout this article, we work over the complex number field $\bC$. 
%%%%%%%%%%%%%%%%%%%%%%%%%%%%%%%%%%%%%%%%%%%%%%%%%%%%%%%%%%%%%%%%%%%%%%%%%%%%%%%%%%%%%%%%%%%%%%%%%%%%%%%%%%
\section{Introduction}\label{1}

Let $X$ be a normal compact complex surface and let $\Gamma$ be a closed subvariety of $X$. 
Then we say that the pair $(X,\Gamma)$ is a compactification of the affine plane $\bC ^2$ if $X \backslash \Gamma$ is biholomorphic to $\bC ^2$. 
Moreover, a compactification $(X,\Gamma)$ of $\bC ^2$ is minimal if $\Gamma$ is irreducible. 

In {\cite{RV60}}, Remmert--Van de Ven proved that if $(X, \Gamma)$ is a minimal compactification of $\bC ^2$ and $X$ is smooth, then $(X, \Gamma ) = (\bP ^2, \text{line})$. 
Hence, we consider minimal compactifications of $\bC ^2$ with singular points in what follows. 
Note that minimal compactification of $\bC ^2$ with at most log canonical singularities are well known. 
Indeed, minimal compactifications of $\bC ^2$ with at most rational double points were studied by Brenton ({\cite{Bre80}}) and Miyanishi--Zhang ({\cite{MZ88}}), etc. 
In {\cite{Koj01}}, Kojima classified minimal compactifications of $\bC ^2$ with at most quotient singular points. 
In {\cite{KT09}}, Kojima--Takahashi determined minimal compactifications of $\bC ^2$ with at most log canonical singular points. 
Moreover, {\cite{KT09}} proved that every minimal compactification of $\bC ^2$ with at most log canonical singular points is a numerically del Pezzo surface of rank one. 
Here, a numerical del Pezzo surface means a normal projective surface with the numerically ample anti-canonical divisor. 

By the list of classifications in {\cite{Koj01,KT09}}, we notice that every minimal compactification of $\bC ^2$ with at most log canonical singularities has only star-shaped singular points. 
In this article, we will classify minimal compactifications of $\bC ^2$ with only star-shaped singularities. 
This result is summarized in the following two theorems: 
%%%%%%%%%%%%%%%%%
\begin{thm}\label{main}
Let $(X,\Gamma )$ be a minimal compactification of the affine plane $\bC ^2$ such that $\Sing (X) \not= \emptyset$, let $\pi :V \to X$ be the minimal resolution, let $D$ be the reduced exceptional divisor of $\pi$, and let $C$ be the proper transform of $\Gamma$ by $\pi$. 
Assume that $X$ has at most star-shaped singular points. 
Then the weighted dual graph of $C+D$ is one of the graphs is given as $(n)$ $(n=1,\dots ,7)$ in Appendix \ref{5}. 
\end{thm}
%%%%%%%%%%%%%%%%%
\begin{thm}\label{main(2)}
Let $V$ be a smooth projective surface, let $C$ be an irreducible curve on $V$, and let $D$ be the reduced divisor on $V$ such that $C \not\subset \Supp (D)$. 
If the weighted dual graph of $C+D$ is one of the graphs listed in Appendix \ref{5}, then $D$ can be contracted. 
\end{thm}
%%%%%%%%%%%%%%%%%
Notice that any minimal compactification $(X,\Gamma)$ of $\bC ^2$ satisfies one of the following conditions: 
%%%%%%%%%%%%%%%%%
\begin{itemize}
\item The anti-canonical divisor $-K_X$ is numerically ample. 
\item The canonical divisor $K_X$ is numerically trivial. 
\item The canonical divisor $K_X$ is numerically ample. 
\end{itemize}
%%%%%%%%%%%%%%%%%
By using several results in the author's previous work {\cite{Saw24}} combined with Theorem \ref{main}, we also obtain the following theorem: 
%%%%%%%%%%%%%%%%%
\begin{thm}\label{main(3)}
Let $(X,\Gamma )$ be a minimal compactification of the affine plane $\bC ^2$ such that $\Sing (X) \not= \emptyset$ and $X$ has at most star-shaped singular points, let $\pi :V \to X$ be the minimal resolution, let $D$ be the reduced exceptional divisor of $\pi$, and let $C$ be the proper transform of $\Gamma$ by $\pi$. 
Then the following assertions hold: 
\begin{enumerate}
\item If the weighted dual graph of $C+D$ is given as (1), (2), (6) or (7) in Appendix \ref{5}, then $-K_X$ is numerically ample. 
\item If the weighted dual graph of $C+D$ is given as (3) in Appendix \ref{5}, then: 
\begin{itemize}
\item $-K_X$ is numerically ample if and only if $0 \le \ell < (n+1)\,d(A)-d(\overline{A})$. 
\item $K_X$ is numerically trivial if and only if $\ell = (n+1)\,d(A)-d(\overline{A})$. 
\item $K_X$ is numerically ample if and only if $(n+1)\,d(A)-d(\overline{A}) < \ell \le d(A)(n\,d(A)-d(\overline{A}))-2$. 
\end{itemize}
\item If the weighted dual graph of $C+D$ is given as (4) in Appendix \ref{5}, then: 
\begin{itemize}
\item $-K_X$ is numerically ample if and only if $0 \le \ell < (n+1)\,d(A)-d(\overline{A})-1$. 
\item $K_X$ is numerically trivial if and only if $\ell = (n+1)\,d(A)-d(\overline{A})-1$ and $s=1$. 
\item $K_X$ is numerically ample if and only if $(n+1)\,d(A)-d(\overline{A})-1 < \ell \le d(A)(n\,d(A)-d(\overline{A}))-2$ or both $\ell = (n+1)\,d(A)+d(\overline{A})-1$ and $s>1$. 
\end{itemize}
\item If the weighted dual graph of $C+D$ is given as (5) in Appendix \ref{5}, then either $-K_X$ or $K_X$ is numerically ample. Moreover: 
\begin{itemize}
\item $-K_X$ is numerically ample if and only if $0 \le \ell < (n+1)\,d(A)-d(\overline{A})-1$. 
\item $K_X$ is numerically ample if and only if $(n+1)\,d(A)-d(\overline{A})-1 \le \ell \le d(A)(n\,d(A)-d(\overline{A}))-2$. 
\end{itemize}
\end{enumerate}
In particular, $K_X$ is numerically trivial if and only if the weighted dual graph of $C+D$ is given as in Figure \ref{fig(1)}, where $A$ is an admissible twig, $A^{\ast}$ is the adjoint of $A$, $m \ge 0$ and $n \ge 2$. 
\end{thm}
%%%%%%%%%%%%%%%%%
\begin{figure}[t]
\begin{center}{\begin{tikzpicture}
\node at (0,0) {\xygraph{
{A^{\ast}} -[l] \circ (-[u] \circ -[r] \cdots ([]!{+(0,.35)} {\overbrace{\quad \qquad \qquad \qquad}^{(n+1)\,d(A)-d(\overline{A})-1}}) -[r] \circ -[r] \circ ([]!{+(0,-.3)} {^{-(m+2)}}) -[r] \bullet ([]!{+(0,.2)} {^{C}}) -[r] \circ -[r] \cdots ([]!{+(0,.35)} {\overbrace{\quad \qquad \qquad \qquad}^{m}}) -[r] \circ ,-[l] {A} -[l] \circ ([]!{+(0,-.3)} {^{-n}})}};
\end{tikzpicture}}\end{center}
\caption{}\label{fig(1)}
\end{figure}
%%%%%%%%%%%%%%%%%
\begin{rem}
In Theorem \ref{main(3)}, if $K_X$ is numerically trivial, then $A \not= [2]$ or $(m,n) \not= (0,2)$ must be satisfied in Figure \ref{fig(1)}. 
Indeed, for an admissible twig $A$ and an integer $n$ with $n \ge 2$, the inequality $(n+1)\,d(A)-d(\overline{A}) -1 \le d(A)(n\,d(A)-d(\overline{A}))-2$ holds; furthermore, the equality in this inequality holds if and only if $A = [2]$ and $n = 2$ (see also Lemmas \ref{lem(4-2-1)} and \ref{lem(4-2-2)}). 
\end{rem}
%%%%%%%%%%%%%%%%%
As an application of Theorem \ref{main(3)}, we can construct many examples of minimal compactifications, whose anti-canonical divisors are numerically ample, of $\bC ^2$ with non-log canonical singularities. 
%%%%%%%%%%%%%%%%%%%%%%%%%%%%%%%%%%%%%%%%%%%%%%%%%%%%%%%%
\subsection*{Organization of article}
In Section \ref{2}, we prepare some basic notations and facts in order to prove the main results. 
%%%%%%%%%%%%%%%%%
In Section \ref{3}, we summarize several results on compactifications of the affine plane. 
Then we mainly refer to {\cite[\S 4]{KT09}} and {\cite{Saw24}}. 
%%%%%%%%%%%%%%%%%
In Section \ref{4}, we will prove Theorems \ref{main}, \ref{main(2)} and \ref{main(3)}. 
%%%%%%%%%%%%%%%%%%%%%%%%%%%%%%%%%%%%%%%%%%%%%%%%%%%%%%%%
\begin{conventions}
A reduced effective divisor $D$ on a smooth projective surface is called an {\it SNC-divisor} if $D$ has only simple normal crossings. 

For any weighted dual graph, a vertex $\circ$ with the weight $m$ corresponds to an $m$-curve (see also the following Notation).  Exceptionally, we omit this weight (resp.\ we omit this weight and use the vertex $\bullet$ instead of $\circ$) if $m=-2$ (resp.\ $m=-1$). 
\end{conventions}
%%%%%%%%%%%%%%%%%%%%%%%%%%%%%%%%%%%%%%%%%%%%%%%%%%%%%%%%
\begin{notation}
We will use the following notations: 
\begin{itemize}
\item $K_X$: the canonical divisor of a normal surface $X$. 
\item $\varphi ^{\ast}(D)$: the total transform of a divisor $D$ by a morphism $\varphi$. 
\item $\psi _{\ast}(D)$: the direct image of a divisor $D$ by a morphism $\psi$. 
\item $(D \cdot D')$: the intersection number of two divisors $D$ and $D'$. 
\item $(D)^2$: the self-intersection number of a divisor $D$. 
\item $D_1 \equiv D_2$: $D_1$ and $D_2$ are numerically equivalent. 
\item $m$-curve: a smooth projective rational curve with self-intersection number $m$. 
\end{itemize}
\end{notation}
%%%%%%%%%%%%%%%%%
\begin{acknowledgment}
The author would like to thank the referee for careful reading this article and suggesting many useful comments. 
\end{acknowledgment}
%%%%%%%%%%%%%%%%%%%%%%%%%%%%%%%%%%%%%%%%%%%%%%%%%%%%%%%%%%%%%%%%%%%%%%%%%%%%%%%%%%%%%%%%%%%%%%%%%%%%%%%%%%
\section{Preliminaries}\label{2}
%%%%%%%%%%%%%%%%%
Let $V$ be a smooth projective surface and let $D$ be a reduced divisor on $V$. 
Letting $D=\sum _{i=1}^nD_i$ be the decomposition of $D$ into irreducible components, we consider the intersection matrix $I(D) := [(D_i \cdot D_j)]_{1 \le i,\,j \le n}$ of $D$. 
In this article, for the weighted dual graph $G$ of $D$, we define the {\it determinant} $d(G)$ of $G$ by the determinant of the matrix $-I(D)$, where we put $d(G) := 1$ if $G = \emptyset$.  
By virtue of the following lemma, we know that $D$ can be contracted if and only if $d(G) > 0$. 
%%%%%%%%%%%%%%%%%
\begin{lem}[{\cite{Mum61,Gra62}}]\label{cont}
With the same notations as above, $D$ can be contracted if and only if $I(D)$ is negative definite. 
\end{lem}
%%%%%%%%%%%%%%%%%
In what follows, we assume that every irreducible component of $D$ is a smooth rational curve. 
In this article, we say that a subgraph $A$ of $G$ is a {\it twig} if it is a linear graph. 
Moreover, for a twig $A$, we write $[a_1,\dots ,a_r]$ as $A$ if the configuration of $A$ is the following: 
%%%%%%%%%%%%%%%%%
\begin{align*}
\xygraph{
\circ ([]!{+(0,-.3)} {^{-a_1}}) -[r] \circ ([]!{+(0,-.3)} {^{-a_2}}) -[r] \cdots ([]!{+(0,-.3)} {}) -[r] \circ ([]!{+(0,-.3)} {^{-a_r}})
}.
\end{align*}
%%%%%%%%%%%%%%%%%
For a positive integer $m$ and an integer $a$, we write $[m *a] := [\underbrace{a,\dots ,a}_{m}]$. 
The following definition is based on {\cite[\S 3]{Fuj82}} (see also {\cite[pp.\ 168--169]{Koj01}}): 
%%%%%%%%%%%%%%%%%
\begin{defin}\label{twig}
Let $A =[a_1,\dots ,a_r]$ be a twig. 
Then the twig $[a_r, a_{r-1}, \ldots , a_1]$ is called the {\it transposal} of $A$ and denoted by ${}^tA$. 
We define also $\overline{A} := [a_2, \dots, a_r]$ and $\underline{A} := [a_1 ,\dots ,a_{r-1}]$, where we put $\overline{A} = \underline{A} = \emptyset$ if $r=1$. 
We say that $A$ is {\it admissible} if $a_i \ge 2$ for any $i=1,\dots ,r$. 
In what follows, we assume that $A$ is admissible. 
We say that $e(A) := d(\overline{A})/d(A)$ is the {\it inductance} of $A$. 
By {\cite[Corollary (3.8)]{Fuj82}}, $e$ defines a one-to-one correspondence from the set of all admissible twigs to the set of rational numbers in the interval $(0,1)$. 
Hence, there exists uniquely an admissible twig ${A}^*$, whose inductance is equal to $1-e({}^tA)$, so that we say that ${A}^*$ is the {\it adjoint} of $A$. 
\end{defin}
%%%%%%%%%%%%%%%%%
We summarize the basic results on twigs: 
%%%%%%%%%%%%%%%%%
\begin{lem}\label{fujita}
Let $A$ be an admissible twig. Then we have: 
\begin{enumerate}
\item $d(\overline{A})d(\underline{A})-d(A)d(\overline{\underline{A}})=1$. 
\item $d(A^{\ast}) = d(A)$ and $d(\overline{A^{\ast}}) = d(A)-d(\underline{A})$. 
\end{enumerate}
\end{lem}
%%%%%%%%%%%%%%%%%
\begin{proof}
In (1), see {\cite[Lemma (3.6)]{Fuj82}}. 
On the other hand, note that $d(A)$ and $d(\overline{A})$ are coprime (see {\cite[Lemma (3.6)]{Fuj82}}). 
Hence, the assertion (2) follows from this result and the definition of the adjoint. 
\end{proof}
%%%%%%%%%%%%%%%%%
Moreover, we will use the following result: 
%%%%%%%%%%%%%%%%%
\begin{lem}\label{adjoint}
Let $A = [a_1,\dots ,a_r]$ and $B = [b_1,\dots ,b_s]$ be admissible twigs, and let $m$ be an integer. 
Then the twig $[m,a_1,\dots ,a_r,1,b_1,\dots ,b_s]$ can be contracted to $[m,1]$ if and only if $B = \underline{A^{\ast}}$. 
\end{lem}
%%%%%%%%%%%%%%%%%
\begin{proof}
See {\cite[Lemma 3.5]{Saw22}}. 
\end{proof}
%%%%%%%%%%%%%%%%%
In what follows, assume further that every irreducible component of $D$ has self-intersection number $\le -2$ and the intersection matrix of $D$ is negative definite. 
Let $D=\sum _iD_i$ be the decomposition of $D$ into irreducible components. 
Since the intersection matrix of $D$ is negative definite, there exists uniquely a $\bQ$-divisor $D^{\natural}=\sum _i \alpha _iD_i$ on $V$ such that $(D_i \cdot K_V+D^{\natural})=0$ for every irreducible component $D_i$ of $D$. 
We note that $D^{\natural}$ is effective (see, e.g., {\cite[Lemma 7.1]{Zar62}}). 
%%%%%%%%%%%%%%%%%
\begin{lem}[cf.\ {\cite[Lemma 2.4]{Saw24}}]\label{alpha}
With the same notations and the assumptions as above, we have the following assertions: 
\begin{enumerate}
\item Let $D_0$ and $D_1$ be irreducible components of $D$ such that $(D_0 \cdot D-D_0) = 1$ and $(D_0 \cdot D_1) = 1$. If $\alpha _0 \ge 1$, then $\alpha _1 \ge 2$. 
\item Let $D_0$, $D_1$ and $D_2$ be irreducible components of $D$ such that $(D_1 \cdot D-D_1) = 2$ and $(D_0 \cdot D_1) = (D_1 \cdot D_2) = 1$. If $\alpha _1 > \alpha _0 > 0$ and $\alpha _1 \ge 1$, then $\alpha _2 > \alpha _1$. 
\end{enumerate}
\end{lem}
%%%%%%%%%%%%%%%%%
\begin{proof}
In (1), for simplicity, we put $m_0 := -(D_0)^2 \ge 2$. 
Since $-(D_0 \cdot K_V) = (D_0 \cdot D^{\natural})$, we then have $-m_0+2 = -\alpha _0m_0 + \alpha _1$. 
Thus, we obtain $\alpha _1 = (\alpha _0-1)m_0+2  \ge 2$. 

In (2), for simplicity, we put $m_1 := -(D_1)^2 \ge 2$. 
Since $-(D_1 \cdot K_V) = (D_1 \cdot D^{\natural})$, we then have $-m_1+2 = -\alpha _1m_1 + \alpha _0 + \alpha _2$. 
Thus, we obtain $\alpha _2 - \alpha _1 = (\alpha _1-1)m_1+2-\alpha _0 -\alpha _1 \ge 2(\alpha _1-1)+2-\alpha _0 - \alpha _1 = \alpha _1-\alpha _0 >0$. Namely, $\alpha _2 > \alpha _1$ holds true. 
\end{proof}
%%%%%%%%%%%%%%%%%%%%%%%%%%%%%%%%%%%%%%%%%%%%%%%%%%%%%%%%%%%%%%%%%%%%%%%%%%%%%%%%%%%%%%%%%%%%%%%%%%%%%%%%%%
\section{Some results on compactifications of the affine plane}\label{3}
%%%%%%%%%%%%%%%%%%%%%%%%%%%%%%%%%%%%%%%%%%%%%%%%%%%%%%%%
\subsection{Minimal normal compactifications of the affine plane}\label{3-1}
%%%%%%%%%%%%%%%%%
In this subsection, we recall a definition and basic results of minimal normal compactifications of $\bC ^2$. 
%%%%%%%%%%%%%%%%%
\begin{defin}
Let $V$ be a smooth projective surface, and let $D$ be an SNC-divisor on $V$. 
Then we say that the pair $(V,D)$ is a {\it minimal normal compactificaion} of $\bC ^2$ if $V \backslash \Supp (D)$ is isomorphic to $\bC ^2$ and $(E \cdot D-E) \ge 3$ for any $(-1)$-curve $E \subseteq \Supp (D)$.  
\end{defin}
%%%%%%%%%%%%%%%%%
Note that minimal normal compactifications of $\bC ^2$ are classified due to Ramanujam and Morrow ({\cite{Ram71,Mor73}}). 
In particular, {\cite{Mor73}} summarizes a list of all minimal normal compactifications of $\bC ^2$. 
In this article, we will use the following result: 
%%%%%%%%%%%%%%%%%
\begin{lem}\label{lem(3-1-1)}
Let $(V,D)$ be a minimal normal compactification of $\bC^2$. Then the following assertions hold: 
\begin{enumerate}
\item $D$ is a rational chain. 
\item If $D$ contains at least three irreducible components, then $D$ contains exactly two irreducible components $D_1$ and $D_2$ such that $(D_1 \cdot D_2) = 1$ and $(D_i)^2 \ge 0$ for $i=1,2$. 
\end{enumerate}
\end{lem}
%%%%%%%%%%%%%%%%%
\begin{proof}
See {\cite{Mor73}}. 
\end{proof}
%%%%%%%%%%%%%%%%%%%%%%%%%%%%%%%%%%%%%%%%%%%%%%%%%%%%%%%%
\subsection{Minimal singular compactifications of the affine plane}\label{3-2}
%%%%%%%%%%%%%%%%
In this subsection, we review some results on minimal (singular) compactifications of the affine plane $\bC ^2$. 
We refer to {\cite{Koj01,KT09,Saw24}}. 
Let $(X,\Gamma )$ be a minimal compactification of $\bC ^2$ such that $\Sing (X) \not= \emptyset$, let $\pi :V \to X$ be the minimal resolution, let $D$ be the exceptional divisor of $\pi$, and let $C$ be the proper transform of $\Gamma$ by $\pi$. 
%%%%%%%%%%%%%%%%
\begin{lem}\label{lem(3-2-1)}
With the same notations as above, the following assertions hold: 
\begin{enumerate}
\item $\sharp \Sing (X) \le 2$ and $C+D$ is an SNC-divisor. Moreover, the dual graph of $C+D$ is as follows: 
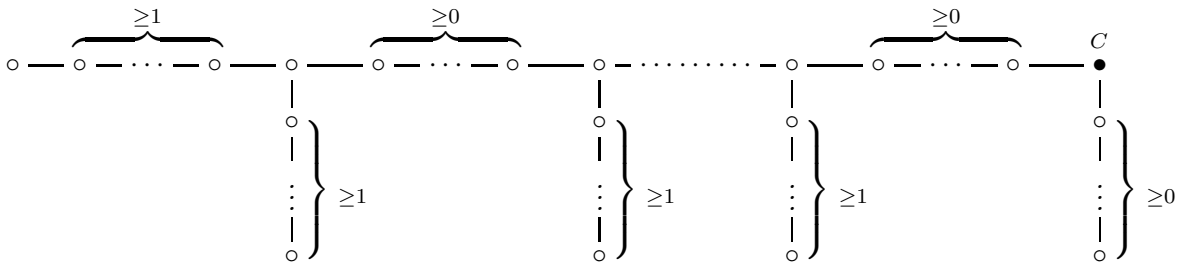
\begin{align*}
\xygraph{
\circ - []!{+(.7,0)} \circ 
- []!{+(.7,0)} \cdots ([]!{+(0,.4)} {\overbrace{\quad \qquad \qquad}^{\ge 1}}) - []!{+(.7,0)} \circ - []!{+(.8,0)} \circ
(- []!{+(0,-.6)} \circ - []!{+(0,-.7)} \vdots ([]!{+(.55,0)} {\left\} \begin{array}{l} \ \\ \ \\ \! ^{\ge 1} \\ \ \end{array} \right.}) - []!{+(0,-.7)} \circ,
- []!{+(.9,0)} \circ - []!{+(.7,0)} \cdots ([]!{+(0,.4)} {\overbrace{\quad \qquad \qquad}^{\ge 0}}) - []!{+(.7,0)} \circ - []!{+(.9,0)} \circ
(- []!{+(0,-.6)} \circ - []!{+(0,-.7)} \vdots ([]!{+(.55,0)} {\left\} \begin{array}{l} \ \\ \ \\ \! ^{\ge 1} \\ \ \end{array} \right.}) - []!{+(0,-.7)} \circ,
-[r] {\cdots \cdots \cdots} -[r] \circ
(- []!{+(0,-.6)} \circ - []!{+(0,-.7)} \vdots ([]!{+(.55,0)} {\left\} \begin{array}{l} \ \\ \ \\ \! ^{\ge 1} \\ \ \end{array} \right.}) - []!{+(0,-.7)} \circ,
- []!{+(.9,0)} \circ - []!{+(.7,0)} \cdots ([]!{+(0,.4)} {\overbrace{\quad \qquad \qquad}^{\ge 0}}) - []!{+(.7,0)} \circ - []!{+(.9,0)} \bullet ([]!{+(0,.2)} {^C})
- []!{+(0,-.6)} \circ - []!{+(0,-.7)} \vdots ([]!{+(.55,0)} {\left\} \begin{array}{l} \ \\ \ \\ \! ^{\ge 0} \\ \ \end{array} \right.}) - []!{+(0,-.7)} \circ
)))
}
\end{align*}
\item If $C$ is not a $(-1)$-curve, then $V$ is the Hirzebruch surface of degree $\ge 2$. 
\end{enumerate}
\end{lem}
%%%%%%%%%%%%%%%%%
\begin{proof}
See {\cite[\S 3]{Saw24}} (see also {\cite[\S 3]{Koj01}} and {\cite[Lemmas 4.1--4.6]{KT09}}). 
\end{proof}
%%%%%%%%%%%%%%%%%
Let $D=\sum _{i=1}^nD_i$ be the decomposition of $D$ into irreducible components. 
Since the intersection matrix of $D$ is negative definite, there exists uniquely an effective $\bQ$-divisor $D^{\natural}=\sum _{i=1}^n \alpha _iD_i$ on $V$ such that $(D_i \cdot K_V+D^{\natural})=0$ for every irreducible component $D_i$ of $D$. 
We say that a divisor $\Delta$ on $X$ is {\it numerically ample} (resp. {\it numerically trivial}) if $(\Delta \cdot B) > 0$ and $(\Delta )^2 > 0$ hold (resp. $(\Delta \cdot B) = 0$ holds) for every curve $B$ on $X$ with Mumford's rational intersection number (see {\cite{Mum61,Sak84}}). 
Then we have the following two lemmas: 
%%%%%%%%%%%%%%%%%
\begin{lem}\label{lem(3-2-2)}
With the same notations as above, assume further that $C$ is a $(-1)$-curve. 
Then the following assertions hold: 
\begin{enumerate}
\item The anti-canonical divisor $-K_X$ is numerically ample if and only if $(D^{\natural} \cdot C)<1$. 
\item The canonical divisor $K_X$ is numerically trivial if and only if $(D^{\natural} \cdot C)=1$. 
\item The canonical divisor $K_X$ is numerically ample if and only if $(D^{\natural} \cdot C)>1$. 
\end{enumerate}
\end{lem}
%%%%%%%%%%%%%%%%%
\begin{proof}
See {\cite[Lemma 2.7]{Saw24}}. 
\end{proof}
%%%%%%%%%%%%%%%%%
\begin{lem}\label{lem(3-2-3)}
With the same notations as above, assume further that the canonical divisor $K_X$ is numerically trivial. 
Then $\alpha _i$ is an integer for every $i$; in other words, $D^{\natural}$ is a $\bZ$-divisor. 
\end{lem}
%%%%%%%%%%%%%%%%%
\begin{proof}
See {\cite[Claim 5.2]{Saw24}}. 
\end{proof}
%%%%%%%%%%%%%%%%%
From now on, assume further that $D$ has a branching component. 
Then $C$ is a $(-1)$-curve by Lemma \ref{lem(3-2-1)} (2). 
Hence, we obtain the contraction $f:V \to V'$ of $C$. 
Since $f_{\ast}(D)$ has a branching component, $(V',f_{\ast}(D))$ is a compactification of $\bC ^2$ but is not minimal normal by Lemma \ref{lem(3-1-1)}. 
Hence, there exists a $(-1)$-curve $C'$ on $\Supp (f_{\ast}(D))$ such that $(C' \cdot f_{\ast}(D)-C') \le 2$. 
Let $g:V' \to V''$ be the contraction of $C'$. 
Since $(g \circ f)_{\ast}(D)$ consists of at least three irreducible components, $(V'',(g \circ f)_{\ast}(D))$ is a compactification of $\bC ^2$ but is not minimal normal (see {\cite[Lemma 4.5]{KT09}}). 
Hence, there exists a $(-1)$-curve $C''$ on $\Supp ((g \circ f)_{\ast}(D))$ such that $(C'' \cdot (g \circ f)_{\ast}(D)-C'') \le 2$. 
Put $D' := f_{\ast}(D) - C'$ and $D'' := g_{\ast}(D') - C''$, and let $D'=\sum _{i=1}^{n-1}D_i'$ (resp. $D''=\sum _{i=1}^{n-2}D_i''$) be the decomposition of $D'$ (resp. $D''$) into irreducible components. 

We will use the following lemma to prove Theorem \ref{main(3)}: 
%%%%%%%%%%%%%%%%%
\begin{lem}\label{lem(3-2-4)}
With the same notations as above, the following assertions hold: 
\begin{enumerate}
\item Assume that $D$ is connected or $D'$ is not connected. 
Then the following assertions hold: 
\begin{itemize}
\item The intersection matrix of $D'$ is negative definite. 
\item If $(C \cdot D^{\natural}) > 1$, then $(C' \cdot {D'}^{\natural}) \ge 1$. 
\item If $(C \cdot D^{\natural}) = 1$, then $(C' \cdot {D'}^{\natural}) \le 1$. 
\item If $(C \cdot D^{\natural}) < 1$, then $(C' \cdot {D'}^{\natural}) < 1$. 
\end{itemize}
Here, ${D'}^{\natural}=\sum _{i=1}^{n-1} \alpha _i'D_i'$ is the effective $\bQ$-divisor on $V'$ such that $(D_i' \cdot K_{V'}+{D'}^{\natural})=0$ for every irreducible component $D_i'$ of $D'$. 
\item Assume that $D$ is not connected and $D'$ is connected. 
Then the following assertions hold: 
\begin{itemize}
\item The intersection matrix of $D''$ is negative definite. 
\item If $(C \cdot D^{\natural}) > 1$, then $(C'' \cdot {D''}^{\natural}) \ge 1$. 
\item If $(C \cdot D^{\natural}) \le 1$, then $(C'' \cdot {D''}^{\natural}) < 1$. 
\end{itemize}
Here, ${D''}^{\natural}=\sum _{i=1}^{n-2} \alpha _i''D_i''$ is the effective $\bQ$-divisor on $V''$ such that $(D_i'' \cdot K_{V''}+{D''}^{\natural})=0$ for every irreducible component $D_i''$ of $D''$. 
\end{enumerate}
\end{lem}
%%%%%%%%%%%%%%%%%
\begin{proof}
See {\cite[Lemmas 4.1, 4.2 and 4.4]{Saw24}}. 
\end{proof}
%%%%%%%%%%%%%%%%%%%%%%%%%%%%%%%%%%%%%%%%%%%%%%%%%%%%%%%%
\subsection{Boundaries of compactifications of the affine plane}\label{3-3}
%%%%%%%%%%%%%%%%%
Let $V$ be a smooth projective surface, let $C$ be a $(-1)$-curve on $V$ and let $D$ be a reduced divisor on $V$ such that every irreducible component of $D$ has self-intersection number $\le -2$. 
Assume that $C+D$ is an SNC-divisor and $(V,C+D)$ is a compactification of $\bC ^2$. 
%%%%%%%%%%%%%%%%%
\begin{lem}\label{lem(3-3-1)}
With the same notations and assumptions as above, the determinant of the intersection matrix of $C+D$ is equal to $-1$. 
\end{lem}
%%%%%%%%%%%%%%%%%
\begin{proof}
It follows from {\cite[Lemma 2.9]{Saw24}}. 
\end{proof}
%%%%%%%%%%%%%%%%%
In what follows, we assume that $D$ consists of exactly two connected components, say $D^{(1)}$ and $D^{(2)}$, and $(C \cdot D^{(1)}) = (C \cdot D^{(2)}) = 1$. 
Hence, there exist irreducible components $D_1$ and $D_2$ of $D^{(1)}$ and $D^{(2)}$ meeting $C$, respectively. 
Assume further $(D_1)^2 = -m$ for some $m \ge 3$ and $(D_2)^2 = -2$. 
Let $f:V \to V'$ be a contraction of $C$. 
Put $C' := f_{\ast}(D_2)$ and $D' := f_{\ast}(D)-C'$. 
Then $(V',C'+D')$ is a compactification of $\bC^2$. 
Let $A$ and $B$ be weighted dual graphs of $D^{(1)}-D_1$ and $D^{(2)}-D_2$, respectively. 
Note that weighted dual graphs of $C+D$ and $C'+D'$ are as follows: 
%%%%%%%%%%%%%%%%%
\begin{align*}
\xygraph{A -[r] \circ ([]!{+(0,.3)} {^{D_1}}) ([]!{+(0,-.3)} {^{-m}}) -[r] \bullet ([]!{+(0,.3)} {^C}) -[r] \circ ([]!{+(0,.3)} {^{D_2}}) -[r] B}
\quad {\overset{f}{\longrightarrow}} \quad
\xygraph{A -[r] \circ ([]!{+(0,.3)} {^{D_1'}}) ([]!{+(0,-.3)} {^{-(m-1)}}) -[r] \bullet ([]!{+(0,.3)} {^{C'}}) -[r] B},
\end{align*}
%%%%%%%%%%%%%%%%%
where $D_1' := f_{\ast}(D_1)$. 
Let $\underline{A}$ and $\overline{B}$ be weighted dual graphs of $D^{(1)}-D_1$ and $D^{(2)}-D_2$ except for all irreducible components meeting $D_1$ and $D_2$, respectively. 
Note that $B$ is not always a twig. 
For example, this situation can occur when the dual graph as in Lemma \ref{lem(3-2-1)} is inverted left to right. 
Then we obtain the following two lemmas: 
%%%%%%%%%%%%%%%%%
\begin{lem}\label{lem(3-3-2)}
With the same notations and assumptions as above, assume further $A \not= \emptyset$ and $B \not= \emptyset$. 
Then $D$ can be contracted if and only if so does $D'$. 
\end{lem}
%%%%%%%%%%%%%%%%%
\begin{proof}
In the ``only if'' part, see {\cite[Lemma 4.4 (1)]{Saw24}}. 
Thus, we shall show the ``if'' part. 
Suppose on the contrary that $D'$ can be contracted but $D$ can not be contracted. 
Then we have $d(A) >0$, $d(B)>0$, $(m-1)d(A)-d(\underline{A}) \ge 1$, and $d(B)-d(\overline{B}) \le -d(B)$. 
By Lemma \ref{lem(3-3-1)}, we have: 
\begin{align}\label{(3-3-1)}
\begin{split}
-1 &= \left|\begin{array}{cccc}d(A)&-d(\underline{A})&0&0 \\ -1&m-1&-1&0 \\ 0&-1&1&-1 \\ 0&0&-d(\overline{B})&d(B)\end{array}\right| \\
&= ((m-1)d(A)-d(\underline{A}))(d(B)-d(\overline{B})) -d(A)d(B) \\
&\le -(m\,d(A)-d(\underline{A}))d(B). 
\end{split}
\end{align}
Meanwhile, we have $(m\,d(A)-d(\underline{A}))d(B) \ge 2$ by virtue of $(m-1)d(A)-d(\underline{A}) \ge 1$, $d(A) \ge 1$ and $d(B) \ge 1$. 
Namely:
\begin{align*}
-2 \ge -(m\,d(A)-d(\underline{A}))d(B). 
\end{align*}
This is a contradiction to (\ref{(3-3-1)}). 
\end{proof}
%%%%%%%%%%%%%%%%%
\begin{lem}\label{lem(3-3-3)}
With the same notations and assumptions as above, assume further $A \not= \emptyset$, $B = \emptyset$ and $m=3$. 
Let $g:V' \to V''$ be the contraction of $C'$, and put $D'' := g_{\ast}(D')-g_{\ast}(D_1')$. 
Then $D$ can be contracted if and only if so does $D''$. 
\end{lem}
%%%%%%%%%%%%%%%%%
\begin{proof}
The ``only if'' part is clear (see also {\cite[Lemma 4.2 (1)]{Saw24}}). 
Thus, we shall show the ``if'' part. 
Note that $d(A) \ge 1$ because $D''$ can be contracted. 
By Lemma \ref{lem(3-3-1)}, we have: 
\begin{align}\label{(3-3-2)}
-1 = \left|\begin{array}{ccc}d(A)&-d(\underline{A})&0 \\ -1&2&-1 \\ 0&-1&1\end{array}\right| =d(A)-d(\underline{A}).  
\end{align}
Hence, we have: 
\begin{align*}
\left|\begin{array}{cc}d(A)&-d(\underline{A}) \\ -1&3\end{array}\right| = 3d(A)-d(\underline{A}) \underset{(\ref{(3-3-2)})}{=} 2d(A)-1 >0
\end{align*}
by virtue of $d(A) \ge 1$. 
This implies that $D$ can be contracted. 
\end{proof}
%%%%%%%%%%%%%%%%%
\begin{rem}
Let the notation be the same as in Lemma \ref{lem(3-3-3)}, and put $C'' := g_{\ast}(D_1')$. 
Then we note that $(V'',C''+D'')$ is a compactificaion of $\bC ^2$ and the weighted dual graph of $C+D$, $C'+D'$ and $C''+D''$ are as follows: 
\begin{align*}
\xygraph{\circ ([]!{+(0,.3)} {^{D_2}}) -[l] \bullet ([]!{+(0,.3)} {^C}) -[l] \circ ([]!{+(0,.3)} {^{D_1}}) ([]!{+(0,-.3)} {^{-3}}) -[l] A}
\quad {\overset{f}{\longrightarrow}} \quad
\xygraph{\bullet ([]!{+(0,.3)} {^{C'}}) -[l] \circ ([]!{+(0,.3)} {^{D_1'}}) -[l] A}
\quad {\overset{g}{\longrightarrow}} \quad
\xygraph{\bullet ([]!{+(0,.3)} {^{C''}}) -[l] A}.
\end{align*}
\end{rem}
%%%%%%%%%%%%%%%%%%%%%%%%%%%%%%%%%%%%%%%%%%%%%%%%%%%%%%%%%%%%%%%%%%%%%%%%%%%%%%%%%%%%%%%%%%%%%%%%%%%%%%%%%%
\section{Proof of main results}\label{4}
%%%%%%%%%%%%%%%%%%%%%%%%%%%%%%%%%%%%%%%%%%%%%%%%%%%%%%%%
\subsection{Proof of Theorem \ref{main(2)}}\label{4-2}
%%%%%%%%%%%%%%%%%
In this subsection, we prove Theorem \ref{main(2)}. 
Let $V$ be a smooth projective surface, let $C$ be an irreducible curve on $V$, and let $D$ be a reduced divisor on $V$ such that $C \not\subset \Supp (D)$. 

%%%%%%%%%%%%%%%%%
We first consider the following two lemmas: 
%%%%%%%%%%%%%%%%%
\begin{lem}\label{lem(4-2-1)}
With the same notations as above, assume further that the weighted dual graph of $D$ looks like that in Figure \ref{fig(4-0-1)}, where $A$ is an admissible twig, $A^{\ast}$ is the adjoint of $A$, $\ell \ge 0$ and $n \ge 2$. 
Then $D$ can be contracted if and only if $\ell \le d(A)(n\,d(A)-d(\overline{A}))-2$. 
\end{lem}
%%%%%%%%%%%%%%%%%
\begin{figure}[t]
\begin{center}{\begin{tikzpicture}
\node at (12,0.75) {\xygraph{
{A^{\ast}} -[l] \circ (-[u] \circ -[r] \cdots ([]!{+(0,.35)} {\overbrace{\quad \qquad \qquad \qquad}^{\ell}}) -[r] \circ  ,-[l] {A} -[l] \circ ([]!{+(0,-.3)} {^{-n}})}};
\end{tikzpicture}}\end{center}
\caption{}\label{fig(4-0-1)}
\end{figure}
%%%%%%%%%%%%%%%%%
\begin{proof}
We set: 
\begin{align*}
d_1 := \frac{n\,d(\underline{A})-d(\underline{\overline{A}})}{n\,d(A)-d(\overline{A})}, \qquad
d_2 := \frac{d(\overline{A^{\ast}})}{d(A^{\ast})}, \qquad
d_3 := \frac{\ell}{\ell + 1}.
\end{align*}
By the configuration of $D$, we know that $D$ can be contracted if and only if: 
\begin{align}\label{(4-2-1)}
2-(d_1+d_2+d_3) > 0. 
\end{align}
We can write $d(A^{\ast}) = d(A)$ and $d(\overline{A^{\ast}}) = d(A) - d(\underline{A})$ by Lemma \ref{fujita} (2). 
Since: 
\begin{align*}
d_2 = \frac{d(A)-d(\underline{A})}{d(A)} = 1 - \frac{d(\underline{A})}{d(A)}, \qquad d_3 = 1 - \frac{1}{\ell +1},
\end{align*}
we have: 
\begin{align*}
2-(d_1+d_2+d_3) = \frac{1}{\ell +1}-\frac{d(\overline{A})d(\underline{A})-d(A)d(\underline{\overline{A}})}{d(A)(n\,d(A)-d(\overline{A}))}
= \frac{1}{\ell +1}-\frac{1}{d(A)(n\,d(A)-d(\overline{A}))}
\end{align*}
by Lemma \ref{fujita} (1). 
Hence, (\ref{(4-2-1)}) holds if and only if: 
\begin{align}\label{(4-2-2)}
\ell < d(A)(n\,d(A)-d(\overline{A}))-1. 
\end{align}
Since $\ell$ and $d(A)(n\,d(A)-d(\overline{A}))-1$ are integers, we know that (\ref{(4-2-2)}) holds if and only if: 
\begin{align*}
\ell \le d(A)(n\,d(A)-d(\overline{A}))-2. 
\end{align*}
The proof of this lemma is thus completed. 
\end{proof}
%%%%%%%%%%%%%%%%%
\begin{lem}\label{lem(4-2-2)}
With the same notations as above, assume further that the weighted dual graph of $C+D$ is given as one of (a) and (b) in Figure \ref{fig(4-0-2)}, 
where $A$ is an admissible twig, $A^{\ast}$ is the adjoint of $A$, $[b_1,\dots ,b_s]$ is an admissible twig with $b_1 \ge 3$, $\underline{B^{\ast}}$ is the adjoint of $[b_1,\dots ,b_s]$ removed the last component, $\ell \ge 0$, $m \ge 0$ and $n \ge 2$. 
Then $D$ can be contracted if and only if $\ell \le d(A)(n\,d(A)-d(\overline{A}))-2$. 
\end{lem}
%%%%%%%%%%%%%%%%%
\begin{figure}[t]
\begin{tikzpicture}
\node at (0,-1.5) {(a)};
\node at (5.5,-2.25) {\xygraph{
{A^{\ast}} -[l] \circ (-[u] \circ -[r] \cdots ([]!{+(0,.35)} {\overbrace{\quad \qquad \qquad \qquad}^{\ell}}) -[r] \circ -[r] \circ ([]!{+(0,-.3)} {^{-b_1}}) -[r] \cdots -[r] \circ ([]!{+(0,-.3)} {^{-b_s}}) -[r] \bullet ([]!{+(0,.2)} {^{C}}) -[r] {\underline{B^{\ast}}},-[l] {A} -[l] \circ ([]!{+(0,-.3)} {^{-n}})}};

\node at (0,-4.5) {(b)};
\node at (7,-5.5) {\xygraph{
{A^{\ast}} -[l] \circ (-[u] \circ -[r] \cdots ([]!{+(0,.35)} {\overbrace{\quad \qquad \qquad \qquad}^{\ell}}) -[r] \circ -[r] \circ ([]!{+(0,-.3)} {^{-b_1}}) -[r] \cdots -[r] \circ ([]!{+(0,-.3)} {^{-b_s}}) -[r] \circ ([]!{+(0,-.3)} {^{-(m+2)}}) (-[r] {\underline{B^{\ast}}},-[u] \bullet ([]!{+(0,.2)} {^{C}}) -[r] \circ -[r] \cdots ([]!{+(0,.35)} {\overbrace{\quad \qquad \qquad \qquad}^{m}}) -[r] \circ ),-[l] {A} -[l] \circ ([]!{+(0,-.3)} {^{-n}})}};
\end{tikzpicture}
\caption{}\label{fig(4-0-2)}
\end{figure}
%%%%%%%%%%%%%%%%%
\begin{proof}
By the assumption, we notice that $V \backslash \Supp (C+D) \simeq \bC ^2$. 
Hence, this lemma follows from Lemmas \ref{lem(3-3-2)} and \ref{lem(3-3-3)} combined with Lemma \ref{lem(4-2-1)}. 
\end{proof}
%%%%%%%%%%%%%%%%%
From now on, we assume further that the weighted dual graph of $D$ is given as $(n)$ $(n=1,\dots ,7)$ in Appendix \ref{5}. 
If the weighted dual graph of $D$ is given as (1) or (2) in Appendix \ref{5}, then $D$ can be clearly contracted to a cyclic quotient singularity. 
On the other hand, if the weighted dual graph of $D$ is given as one of (3), (4) and (5) in Appendix \ref{5}, then $D$ can be contracted by Lemmas \ref{lem(4-2-1)} and \ref{lem(4-2-2)}. 
Moreover, if the weighted dual graph of $D$ is given as (6) or (7) in Appendix \ref{5}, then $D$ can be contracted because $-(D_0)^2 \ge (D_0 \cdot D-D_0)$, where $D_0$ is the branching component of $D$. 

The proof of Theorem \ref{main(2)} is thus completed. 
%%%%%%%%%%%%%%%%%%%%%%%%%%%%%%%%%%%%%%%%%%%%%%%%%%%%%%%%
\subsection{Proof of Theorem \ref{main}}\label{4-1}
%%%%%%%%%%%%%%%%%
In this subsection, we prove Theorem \ref{main}. 
Let $(X,\Gamma)$ be a minimal compactification of $\bC ^2$ such that $\Sing (X) \not= \emptyset$, let $\pi :V \to X$ be the minimal resolution, let $D$ be the reduced exceptional divisor of $\pi$, and let $C$ be the proper transform of $\Gamma$ by $\pi$. 
Assume that $X$ has only star-shaped singular points; in other words, every connected component of $D$ has at most one branching component. 

%%%%%%%%%%%%%%%%%
If every connected component of $D$ has no branching component (in other words, every connected component of $D$ is a rational chain), then $X$ has only cyclic quotient singularities. 
By virtue of {\cite[Theorem 1.1]{Koj01}}, we know that the weighted dual graph of $C+D$ is given as one of (1), (2), (3) and (6) in Appendix \ref{5}. 

%%%%%%%%%%%%%%%%%
In what follows, we thus assume that there exists a connected component of $D$ with a branching component. 
Hence, by Lemma \ref{lem(3-2-1)} $C$ is a $(-1)$-curve and the dual graph of $C+D$ is one of (A) and (B) in Figure \ref{fig(4-1)}. 
Let $D = D^{(1)}+D^{(2)}$ be the decomposition of $D$ into connected components such that $D^{(1)}$ is not a rational chain, where we consider $D^{(2)}=0$ if $D$ is connected. 
We consider the cases (A) and (B) separately. 
%%%%%%%%%%%%%%%%%
\begin{figure}[t]
\begin{center}{\begin{tikzpicture}
\node at (0,1) {(A)};
\node at (5.5,0) {\xygraph{
\circ -[l] \cdots ([]!{+(0,.35)} {\overbrace{\quad \qquad \qquad \qquad}^{\ge 1}}) -[l] \circ -[l] \circ (-[u] \circ -[r] \cdots ([]!{+(0,.35)} {\overbrace{\quad \qquad \qquad \qquad}^{\ge 1}}) -[r] \circ -[r] \bullet ([]!{+(0,.2)} {^{C}}) -[r] \circ -[r] \cdots ([]!{+(0,.35)} {\overbrace{\quad \qquad \qquad \qquad}^{\ge 0}}) -[r] \circ,-[l] \circ -[l] \cdots ([]!{+(0,.35)} {\overbrace{\quad \qquad \qquad \qquad}^{\ge 2}}) -[l] \circ}};

\node at (0,-2.5) {(B)};
\node at (5.5,-3.5) {\xygraph{
\circ -[l] \cdots ([]!{+(0,.35)} {\overbrace{\quad \qquad \qquad \qquad}^{\ge 1}}) -[l] \circ -[l] \circ (-[u] \circ -[r] \cdots ([]!{+(0,.35)} {\overbrace{\quad \qquad \qquad \qquad}^{\ge 0}}) -[r] \circ -[r] \circ (-[r] \circ -[r] \cdots ([]!{+(0,.35)} {\overbrace{\quad \qquad \qquad \qquad}^{\ge 1}}) -[r] \circ,-[u] \bullet ([]!{+(0,.2)} {^{C}}) -[r] \circ -[r] \cdots ([]!{+(0,.35)} {\overbrace{\quad \qquad \qquad \qquad}^{\ge 0}}) -[r] \circ ),-[l] \circ -[l] \cdots ([]!{+(0,.35)} {\overbrace{\quad \qquad \qquad \qquad}^{\ge 2}}) -[l] \circ}};
\end{tikzpicture}}\end{center}

\caption{}\label{fig(4-1)}
\end{figure}
%%%%%%%%%%%%%%%%%
\medskip

\noindent
{\bf Case (A):} 
In this case, $D^{(1)}$ contains exactly one branching component, say $D_0$. 
Moreover, $(D_0 \cdot D^{(1)}-D_0) = 3$ holds. 
Let $D^{(1)}-D_0 = T_1+T_2+T_3$ be the decomposition of $D^{(1)}-D_0$ into connected components such that $(C \cdot T_1) = 1$. 

We shall first consider two weighted dual graphs of $T_2$ and $T_3$. 
Let $\widetilde{f} : V \to \widetilde{V}$ be a sequence of contractions of all (smoothly) contractible components in $\Supp (C+T_1+D^{(2)})$, and set $\widetilde{C} := \widetilde{f}_{\ast}(D_0)$ and $\widetilde{D} := \widetilde{f}_{\ast}(D-D_0)$. 
Then $\widetilde{C}+\widetilde{D}$ is a rational chain, and $\widetilde{D}$ has exactly two connected components. 
Moreover, $(\widetilde{V},\widetilde{C}+\widetilde{D})$ is a compactification of $\bC ^2$. 
Hence, we have the contraction $\widetilde{\pi}:\widetilde{V} \to \widetilde{X}$ of $\widetilde{D}$, so that $(\widetilde{X},\widetilde{\pi}_{\ast}(\widetilde{C}))$ is a minimal compactification of $\bC ^2$ with exactly two cyclic quotient singular points. 
By virtue of {\cite[Theorem 1.1]{Koj01}}, we know that the weighted dual graph of $\widetilde{C}+\widetilde{D}$ is as follows: 
\begin{align*}
\xygraph{{A^{\ast}} -[l] \bullet ([]!{+(0,.2)} {^{\widetilde{C}}}) -[l] {A} -[l] \circ ([]!{+(0,-.3)} {^{-n}})}
\end{align*}
where $A$ is an admissible twig, $A^{\ast}$ is the adjoint of $A$ and $n \ge 2$. 
The weighted dual graphs of $T_2$ and $T_3$ are thus determined. 

Next, we consider two weighted dual graphs of $T_1$ and $D^{(2)}$. 
Let $A^{(1)}$ and $A^{(2)}$ be the weighted dual graphs of $T_1$ and $D^{(2)}$, respectively. 
Note that $A^{(1)}$ is an admissible twig, and so is $A^{(2)}$ provided that $A^{(2)} \not= \emptyset$. 
Noting that $-(D_0)^2 \ge 2$, we consider the following two subcases separately. 
%%%%%%%%%%%%%%%%%
\smallskip

\noindent
{\bf Subcase (A)-1:} $-(D_0)^2 = 2$. 
If $D^{(2)} = 0$, then $A^{(1)} = [\ell * 2]$ for some non-negative integer $\ell$ because the intersection matrix of $C+D$ is not negative definite (see, e.g., {\cite[Lemma 1.4]{Zha88}}). 
We note $0 \le \ell \le d(A)(n\,d(A)-d(\overline{A}))-2$ by Lemma \ref{lem(4-2-1)} because $D$ can be contracted. 
Therefore, the weighted dual graph of $C+D$ is given as (3) in Appendix \ref{5}. 
From now on, we assume $D^{(2)} \not= 0$. 
Then $A^{(1)} \not= [\ell *2]$ for any non-negative integer $\ell$. 
Indeed, otherwise, we obtain the following sequence of contractions of (smoothly) contractible components in $\Supp (C+D)$: 
\begin{align*}
\xygraph{
(- []!{+(0,.5)} \circ -[r] \cdots ([]!{+(0,.35)} {\overbrace{\quad \qquad \qquad \qquad}^{\ell}}) -[r] \circ -[r] \bullet ([]!{+(0,.2)} {^{C}}) -[r] \circ ([]!{+(0,-.3)} {^{-b_1'}}) -[r] \cdots -[r] \circ ([]!{+(0,-.3)} {^{-b_{s'}'}}),- []!{+(0,-.5)} \circ (-[l] \cdots,-[r] \cdots))}
\longrightarrow
\xygraph{
(- []!{+(0,.5)} \circ ([]!{+(0,.3)} {^{-b_1'+\ell +1}}) -[r] \circ ([]!{+(0,-.3)} {^{-b_2'}}) -[r] \cdots -[r] \circ ([]!{+(0,-.3)} {^{-b_{s'}'}}),- []!{+(0,-.5)} \bullet (-[l] \cdots,-[r] \cdots))}
\end{align*}
This implies that there exists a minimal normal compactification of $\bC ^2$ such that this boundary has at least three irreducible components; moreover, it is not a rational chain or contains only one irreducible component satisfying the self-intersection number $\ge 0$. 
It is a contradiction to Lemma \ref{lem(3-1-1)}. 
Hence, we can write $A^{(1)} = [\ell *2 , b_1,\dots ,b_s]$ for some non-negative integer $\ell$, where $[b_1,\dots ,b_s]$ is an admissible twig with $b_1 \ge 3$. 
Let $[b_1',\dots ,b_{s'}']$ be the adjoint of $[b_1,\dots ,b_s]$ (for the definition, see Definition \ref{twig}). 
Note that the weighted dual graph of $T_1+C+D^{(2)}$ is the twig $[\ell *2,b_1,\dots ,b_s,1,A^{(2)}]$. 
Since $[\ell *2,b_1,\dots ,b_s,1,A^{(2)}]$ can be contracted to $[\ell *2,1]$, we know $A^{(2)} = [b_1',\dots ,b_{s'-1}']$ by Lemma \ref{adjoint}. 
We note $0 \le \ell \le d(A)(n\,d(A)-d(\overline{A}))-2$ by Lemma \ref{lem(4-2-2)} because $D$ can be contracted. 
Therefore, the weighted dual graph of $C+D$ is given as (4) in Appendix \ref{5}. 
%%%%%%%%%%%%%%%%%
\smallskip

\noindent
{\bf Subcase (A)-2:} $-(D_0)^2 \ge 3$. 
Let $D_2$ be the irreducible component of $T_2$ meeting $D_0$. 
We put $b_1 := -(D_0)^2$ and $a := -(D_2)^2$, and we write $A^{(1)} = [b_2,\dots ,b_s]$. 
Let $[b_1',\dots ,b_{s'}']$ be the adjoint of $[b_1,\dots ,b_s]$ (for the definition, see Definition \ref{twig}). 
Note that the weighted dual graph $D_2+D_0+T_1+C+D^{(2)}$ is the twig $[a,b_1,\dots ,b_s,1,A^{(2)}]$. 
Since $[a,b_1,\dots ,b_s,1,A^{(2)}]$ can be contracted to $[a,1]$, we know $A^{(2)} = [b_1',\dots ,b_{s'-1}']$ by Lemma \ref{adjoint}. 
Therefore, we know that the weighted dual graph of $C+D$ is given as (6) in Appendix \ref{5}. 
%%%%%%%%%%%%%%%%%
\medskip

\noindent
{\bf Case (B):}
Let $m$ be the number of irreducible components of $D^{(2)}$, and let $D_1$ be the irreducible component of $T_1$ meeting $C$. 
Then the weighted dual graph of $D^{(2)}$ is the twig $[m*2]$ by Lemma \ref{lem(3-2-1)}. 
Note that $D_1+C+D^{(2)}$ is a rational chain. 
Since $D_1+C+D^{(2)}$ can be contracted to a single $(-1)$-curve, we know that its weighted dual graph is the twig $[m+2,1,m*2]$. 
Let $\widetilde{f} : V \to \widetilde{V}$ be a sequence of contractions of all (smoothly) contractible components in $\Supp (C+D^{(2)})$, and set $\widetilde{C} :=\widetilde{f}_{\ast}(D_1)$ and $\widetilde{D} := \widetilde{f}_{\ast}(D-D_1)$. 
Then $(\widetilde{V},\widetilde{C}+\widetilde{D})$ is a compactification of $\bC ^2$. 
If the weighted dual graph of $\widetilde{C}+\widetilde{D}$ is  as in Case (A), then the weighted dual graph of $C+D$ is given as (5) or (7) in Appendix \ref{5}. 
Here, if the weighted dual graph of $C+D$ is given as (5), we note $0 \le \ell \le d(A)(n\,d(A)-d(\overline{A}))-2$ by Lemma \ref{lem(4-2-2)} because $D$ can be contracted. 
In what follows, we assume otherwise. Then the weighted dual graph of $\widetilde{C}+\widetilde{D}$ is as follows: 
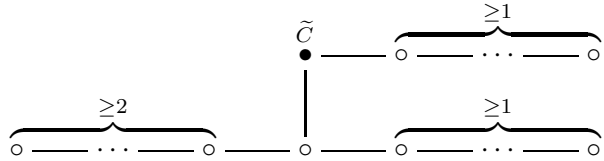
\begin{align*}
\xygraph{
\circ -[l] \cdots ([]!{+(0,.35)} {\overbrace{\quad \qquad \qquad \qquad}^{\ge 1}}) -[l] \circ -[l] \circ (-[u] \bullet ([]!{+(0,.2)} {^{\widetilde{C}}}) -[r] \circ -[r] \cdots ([]!{+(0,.35)} {\overbrace{\quad \qquad \qquad \qquad}^{\ge 1}}) -[r] \circ,-[l] \circ -[l] \cdots ([]!{+(0,.35)} {\overbrace{\quad \qquad \qquad \qquad}^{\ge 2}}) -[l] \circ}
\end{align*}
Hence, we have the contraction $\widetilde{\pi}:\widetilde{V} \to \widetilde{X}$ of $\widetilde{D}$, so that $(\widetilde{X},\widetilde{\pi}_{\ast}(\widetilde{C}))$ is a minimal compactification of $\bC ^2$ with exactly two cyclic quotient singular points. 
By virtue of {\cite[Theorem 1.1]{Koj01}}, we know that the weighted dual graph of $\widetilde{C}+\widetilde{D}$ is as follows: 
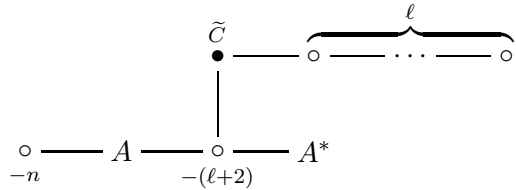
\begin{align*}
\xygraph{{A^{\ast}} -[l] \circ ([]!{+(0,-.3)} {^{-(\ell +2)}}) (-[u] \bullet ([]!{+(0,.2)} {^{\widetilde{C}}}) -[r] \circ -[r] \cdots ([]!{+(0,.35)} {\overbrace{\quad \qquad \qquad \qquad}^{\ell}}) -[r] \circ,-[l] {A} -[l] \circ ([]!{+(0,-.3)} {^{-n}}))}
\end{align*}
where $A$ is an admissible twig, $A^{\ast}$ is the adjoint of $A$, $\ell \ge 0$ and $n \ge 2$. 
Therefore, we know that the weighted dual graph of $C+D$ is given as (7) in Appendix \ref{5} with $b_1 = \ell +2$ and $s=1$. 
%%%%%%%%%%%%%%%%%
\medskip

The proof of Theorem \ref{main} is thus completed. 
%%%%%%%%%%%%%%%%%%%%%%%%%%%%%%%%%%%%%%%%%%%%%%%%%%%%%%%%
\subsection{Proof of Theorem \ref{main(3)}}\label{4-3}
%%%%%%%%%%%%%%%%%
In this subsection, we prove Theorem \ref{main(3)}. 
Let $(X,\Gamma )$ be a minimal compactification of $\bC ^2$ such that $\Sing (X) \not= \emptyset$ and $X$ has at most star-shaped singular points, let $\pi :V \to X$ be the minimal resolution, let $D$ be the reduced exceptional divisor of $\pi$, and let $C$ be the proper transform of $\Gamma$ by $\pi$. 
We first determine the weighted dual graph of $C+D$ when $K_X$ is numerically trivial. 
In other words, we shall show the following lemma: 
%%%%%%%%%%%%%%%%%
\begin{lem}\label{lem(4-3-1)}
With the same notations as above, $K_X$ is numerically trivial if and only if the weighted dual graph of $C+D$ looks like that in Figure \ref{fig(1)}. 
\end{lem}
%%%%%%%%%%%%%%%%%
We shall show the above lemma. 
Let $D=\sum _iD_i$ be the decomposition of $D$ into irreducible components. 
Since the intersection matrix of $D$ is negative definite, there exists uniquely an effective $\bQ$-divisor $D^{\natural}=\sum _i \alpha _iD_i$ on $V$ such that $(D_i \cdot K_V+D^{\natural})=0$ for every irreducible component $D_i$ of $D$. 

Assume that $K_X$ is numerically trivial. 
By the assumption and {\cite[Theorem 1.1 (1)]{KT09}}, we know that $X$ has a singular point, which is worse than a log canonical singular point. 
In particular, $D$ contains a connect component, which is not a rational chain. 
Hence, the weighted dual graph of $C+D$ does not look like (1) nor (2) in Appendix \ref{5}. 
We consider the following four cases separately. 
%%%%%%%%%%%%%%%%%
\medskip

\noindent
{\bf Case 1:} The weighted dual graph of $C+D$ is given as (3) in Appendix \ref{5}. 
%%%%%%%%%%%%%%%%%
\begin{claim}\label{claim(4-1)}
$\ell = (n+1)\,d(A)-d(\overline{A})$. 
\end{claim}
%%%%%%%%%%%%%%%%%
\begin{proof}
%We write $A = [a_1,\dots ,a_r]$ and $A^{\ast} = [a_1',\dots ,a_{r'}']$. 
%We set: 
%\begin{align*}
%d := 
%\left| \begin{array}{cccc|c|ccccc|ccccccc}
%2&-1&&&&&&& \\
%-1&\ddots&\ddots&&&&&&& \\
%&\ddots&\ddots&-1&&&&&& \\ 
%&&-1&2&-1&&&&&& \\ \hline
%&&&-1&2&-1&&&&&-1 \\ \hline
%&&&&-1&a_r&-1&&& \\ 
%&&&&&-1&\ddots&\ddots&& \\ 
%&&&&&&\ddots&\ddots&\ddots& \\ 
%&&&&&&&-1&a_1&-1 \\ 
%&&&&&&&&-1&n \\ \hline
%&&&&-1&&&&&&a_1'&-1 \\
%&&&&&&&&&&-1&\ddots&\ddots \\
%&&&&&&&&&&&\ddots&\ddots&-1 \\
%&&&&&&&&&&&&-1&a_{r'}' \\
%\end{array}\right| 
%\end{align*}
%and: 
%\begin{align*}
%d' := 
%\left| \begin{array}{c|cccc|c|ccccc|ccccccc}
%0&-1&&&&&&&&& \\
%\vdots&2&-1&&&&&&&&& \\
%\vdots&-1&\ddots&\ddots&&&&&&& \\ 
%\vdots&&\ddots&\ddots&-1&&&&&& \\ 
%0&&&-1&2&-1&&&&&& \\ \hline
%0&&&&-1&2&-1&&&&&-1 \\ \hline
%a_r-2&&&&&-1&a_r&-1&&& \\ 
%\vdots&&&&&&-1&\ddots&\ddots&& \\ 
%\vdots&&&&&&&\ddots&\ddots&\ddots& \\ 
%a_1-2&&&&&&&&-1&a_1&-1 \\ 
%n-2&&&&&&&&&-1&n& \\ \hline
%a_1'-2&&&&&-1&&&&&&a_1'&-1 \\
%\vdots&&&&&&&&&&&-1&\ddots&\ddots \\
%\vdots&&&&&&&&&&&&\ddots&\ddots&-1 \\
%a_{r'}'-2&&&&&&&&&&&&&-1&a_{r'}' \\
%\end{array}\right| 
%\end{align*}
Considering the equations $\{ (D_i \cdot D^{\natural}) = (D_i \cdot -K_V)\}_i$, we obtain the linear simultaneous equation for $\{ \alpha _i\} _i$. 
Focus on the solution of this linear simultaneous equation corresponding to the irreducible component meeting $C$. 
By the Cramer formula and Lemma \ref{lem(3-2-2)} (2), we know $1 = d'/d$, where: 
\begin{align*}
d= \left|\begin{array}{cccc}2&-1&-1&-1 \\ 
-n\,d(\underline{A})+d(\underline{\overline{A}})&n\,d(A)-d(\overline{A})&0&0\\
-d(\overline{A^{\ast}})&0&d(A^{\ast})&0\\
-\ell&0&0&\ell +1 \end{array}\right|
\end{align*}
and: 
\begin{align*}
d' = \left|\begin{array}{ccc}0&-1&-1 \\ 
(n\,d(A)-d(\overline{A}))-(n\,d(\underline{A})-d(\underline{\overline{A}}))-1&n\,d(A)-d(\overline{A})&0\\
d(A^{\ast})-d(\overline{A^{\ast}})-1&0&d(A^{\ast}) \end{array}\right| .
\end{align*}
Namely, $d=d'$. 
By noting Lemma \ref{fujita}, we have: 
\begin{align}\label{(4-3-1)}
\begin{split}
d &= d(A)(n\,d(A)-d(\overline{A})) - (\ell +1)(d(\overline{A})d(\underline{A})-d(A)d(\underline{\overline{A}})) \\
&= d(A)(n\,d(A)-d(\overline{A})) - (\ell +1). 
\end{split}
\end{align}
Similarly, we have: 
\begin{align}\label{(4-3-2)}
\begin{split}
d' &= (d(A)-1)(n\,d(A)-d(\overline{A})) -d(A) - (d(\overline{A})d(\underline{A})-d(A)d(\underline{\overline{A}}))\\
&= (d(A)-1)(n\,d(A)-d(\overline{A})) -d(A) -1. 
\end{split}
\end{align}
Thus, we obtain $\ell = (n+1)\,d(A)-d(\overline{A})$. 
\end{proof}
%%%%%%%%%%%%%%%%%
Therefore, the weighted dual graph of $C+D$ is given as in Figure \ref{fig(1)} with $m=0$. 
%%%%%%%%%%%%%%%%%
\medskip

\noindent
{\bf Case 2:} The weighted dual graph of $C+D$ is given as (4) in Appendix \ref{5}. 
In this case, we know $\underline{B^{\ast}} = [m*2]$ for some positive integer $m$ by using Lemmas \ref{lem(3-2-2)} (2) and \ref{lem(3-2-3)}. 
Since $[2,b_1,\dots ,b_s,1,m*2]$ can be contracted $[2,1]$ by considering the weighted dual graph of $C+D$, we obtain $s = 1$ and $[b_1] = [m+2]$. 
%%%%%%%%%%%%%%%%%
\begin{claim}\label{claim(4-2)}
$\ell = (n+1)\,d(A)-d(\overline{A})-1$. 
\end{claim}
%%%%%%%%%%%%%%%%%
\begin{proof}
%We write $A = [a_1,\dots ,a_r]$ and $A^{\ast} = [a_1',\dots ,a_{r'}']$. 
%We set: 
%\begin{align*}
%d := 
%\left| \begin{array}{ccccc|c|ccccc|ccccccc}
%m+2&-1&&&&&&&&& \\
%-1&2&-1&&&&&&&& \\
%&-1&\ddots&\ddots&&&&&&& \\
%&&\ddots&\ddots&-1&&&&&& \\ 
%&&&-1&2&-1&&&&&& \\ \hline
%&&&&-1&2&-1&&&&&-1 \\ \hline
%&&&&&-1&a_r&-1&&& \\ 
%&&&&&&-1&\ddots&\ddots&& \\ 
%&&&&&&&\ddots&\ddots&\ddots& \\ 
%&&&&&&&&-1&a_1&-1 \\ 
%&&&&&&&&&-1&n \\ \hline
%&&&&&-1&&&&&&a_1'&-1 \\
%&&&&&&&&&&&-1&\ddots&\ddots \\
%&&&&&&&&&&&&\ddots&\ddots&-1 \\
%&&&&&&&&&&&&&-1&a_{r'}' \\
%\end{array}\right| 
%\end{align*}
%and: 
%\begin{align*}
%d' := 
%\left| \begin{array}{c|cccc|c|ccccc|ccccccc}
%m&-1&&&&&&&&& \\ \hline
%0&2&-1&&&&&&&& \\
%\vdots&-1&\ddots&\ddots&&&&&&& \\
%\vdots&&\ddots&\ddots&-1&&&&&& \\ 
%0&&&-1&2&-1&&&&&& \\ \hline
%0&&&&-1&2&-1&&&&&-1 \\ \hline
%a_r-2&&&&&-1&a_r&-1&&& \\ 
%\vdots&&&&&&-1&\ddots&\ddots&& \\ 
%\vdots&&&&&&&\ddots&\ddots&-1& \\ 
%a_1-2&&&&&&&&-1&a_1&-1 \\ 
%n-2&&&&&&&&&-1&n \\ \hline
%a_1'-2&&&&&-1&&&&&&a_1'&-1 \\
%\vdots&&&&&&&&&&&-1&\ddots&\ddots \\
%\vdots&&&&&&&&&&&&\ddots&\ddots&-1 \\
%a_{r'}'-2&&&&&&&&&&&&&-1&a_{r'}' \\
%\end{array}\right| 
%\end{align*}
Considering the equations $\{ (D_i \cdot D^{\natural}) = (D_i \cdot -K_V)\}_i$, we obtain the linear simultaneous equation for $\{ \alpha _i\} _i$. 
Focus on the solution of this linear simultaneous equation corresponding to the irreducible component meeting $C$ with self-intersection number $-(m+2)$. 
By the Cramer formula and Lemma \ref{lem(3-2-2)} (2), we know $1 = d'/d$, where: 
\begin{align*}
d= \left|\begin{array}{cccc}2&-1&-1&-1 \\ 
-n\,d(\underline{A})+d(\underline{\overline{A}})&n\,d(A)-d(\overline{A})&0&0\\
-d(\overline{A^{\ast}})&0&d(A^{\ast})&0\\
-(m+2)\ell +(\ell -1)&0&0&(m+2)(\ell +1)-\ell \end{array}\right|.
\end{align*}
Moreover, $d' = d_1' + d_2'$ by using the cofactor expansion of the determinant, where: 
\begin{align*}
d_1' := m\left|\begin{array}{cccc}2&-1&-1&-1 \\ 
-n\,d(\underline{A})+d(\underline{\overline{A}})&n\,d(A)-d(\overline{A})&0&0\\
-d(\overline{A^{\ast}})&0&d(A^{\ast})&0\\
-\ell&0&0&\ell +1 \end{array}\right|
\end{align*}
and: 
\begin{align*}
d_2' := \left|\begin{array}{ccc}0&-1&-1 \\ 
(n\,d(A)-d(\overline{A}))-(n\,d(\underline{A})-d(\underline{\overline{A}}))-1&n\,d(A)-d(\overline{A})&0\\
d(A^{\ast})-d(\overline{A^{\ast}})-1&0&d(A^{\ast}) \end{array}\right| . 
\end{align*}
Namely, $d=d_1'+d_2'$. 
Here, we have: 
\begin{align*}
d = (m+1)d(A)(n\,d(A)-d(\overline{A})) -m(\ell + 1) - (\ell +2).
\end{align*}
by the similar argument to Claim \ref{claim(4-1)}. 
On the other hand, by (\ref{(4-3-1)}) and (\ref{(4-3-2)}) we have: 
\begin{align*}
d_1'+d_2' &= m\{ d(A)(n\,d(A)-d(\overline{A})) -(\ell + 1)\} +\{(d(A)-1)(n\,d(A)-d(\overline{A})) -d(A) -1\} \\
&= (m+1)d(A)(n\,d(A)-d(\overline{A})) -m(\ell + 1) -(n\,d(A)-d(\overline{A})) -d(A) -1.
\end{align*}
Thus, we obtain $\ell = (n+1)\,d(A)-d(\overline{A})-1$. 
\end{proof}
%%%%%%%%%%%%%%%%%
Therefore, the weighted dual graph of $C+D$ is given as in Figure \ref{fig(1)} with $m \ge 1$. 
%%%%%%%%%%%%%%%%%
\medskip

\noindent
{\bf Case 3:} The weighted dual graph of $C+D$ is given as (5) or (7) in Appendix \ref{5}. 
Let $D = D^{(1)}+D^{(2)}$ be the decomposition of $D$ into connected components such that $D^{(1)}$ is not a rational chain, where we consider $D^{(2)}=0$ if $D$ is connected. 
There exists a unique irreducible component $D_0^{(1)}$ of $D^{(1)}$ meeting $C$. 
Let $T^{(1)}$ be the connected component, which is a rational chain, of $D^{(1)}-D_0^{(1)}$. 
Note that the weighted dual graph of $T^{(1)}$ is an admissible twig $\underline{B^{\ast}}$. 
Let $T^{(1)} = \sum _{i=1}^{s'}D_i^{(1)}$ be the decomposition of $T^{(1)}$ into irreducible components such that the weighted dual graph of $C+D_0^{(1)}+T^{(1)}$ is as follows: 
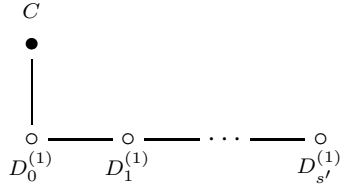
\begin{align*}
\xygraph{\circ ([]!{+(0,-.3)} {^{D_{s'}^{(1)}}}) -[l] \cdots -[l] \circ ([]!{+(0,-.3)} {^{D_1^{(1)}}}) -[l] \circ ([]!{+(0,-.3)} {^{D_0^{(1)}}})
-[u] \bullet \circ ([]!{+(0,.3)} {^C})}
\end{align*}

For $i=0,1,\dots ,s'$, let $\alpha _i^{(1)}$ be the coefficient of $D_i^{(1)}$ in $D^{\natural}$. 
Then we know $\alpha _0^{(1)} = 1$ by Lemmas \ref{lem(3-2-2)} (2) and \ref{lem(3-2-3)}. 
Meanwhile, since $\alpha _{s'}^{(1)} \ge 1$ by virtue of Lemma \ref{lem(3-2-3)}, we have $\alpha _0^{(1)} > 1$ by Lemma \ref{alpha}. 
This is a contradiction. 
Therefore, this case does not take place. 
%%%%%%%%%%%%%%%%%
\medskip

\noindent
{\bf Case 4:} The weighted dual graph of $C+D$ is given as (6) in Appendix \ref{5}. 
In this case, by a similar argument to Case 2 we know $\underline{B^{\ast}} = [m*2]$ for some positive integer $m$ and both $s=1$ and $[b_1] = [m+2]$. 
However, this implies that every connected component of $D$ is a rational chain. 
This is a contradiction. 
Therefore, this case does not take place. 
%%%%%%%%%%%%%%%%%
\medskip

Summary up, the weighted dual graph of $C+D$ looks like that in Figure \ref{fig(1)} provided that $K_X \equiv 0$. 
Conversely, assume that the weighted dual graph of $C+D$ looks like that in Figure \ref{fig(1)}. 
Then we know $(D^{\natural} \cdot C) = 1$ by a similar argument to proof of Claims \ref{claim(4-1)} and \ref{claim(4-2)}.  
Hence, $K_X \equiv 0$ by Lemma \ref{lem(3-2-2)} (2). 
The proof of Lemma \ref{lem(4-3-1)} is thus completed. 

In what follows, we prove Theorem \ref{main(3)}. 
%%%%%%%%%%%%%%%%%
\begin{proof}[Proof of Theorem \ref{main(3)} (2)]
By assumption, note that the weighted dual graph of $C+D$ is given as (3) in Appendix \ref{5}. 
We shall focus on the non-negative integer $\ell$. 
If $\ell = (n+1)d(A) - d(\overline{A})$, then $K_X$ is numerically trivial by Lemma \ref{lem(4-3-1)}. 
From now on, we thus consider the remaining cases. 
We consider the following two cases separately. 
\medskip

\noindent
{\bf Case 1:} $0 \le \ell < (n+1)\,d(A)-d(\overline{A})$. 
In this case, there exists a sequence $\widetilde{f} :\widetilde{V} \to V$ of blowing-ups at a point and its infinitely near points on $C$ such that the weighted dual graph of $\widetilde{C} + \widetilde{D} := \widetilde{f}^{-1}_{\ast}(C+D)_{\rm red.}$ looks like that in Figure \ref{fig(1)}, where $\widetilde{C}$ corresponds to the vertex with label $C$ in Figure \ref{fig(1)}. 
By Lemmas \ref{lem(4-2-1)} and \ref{lem(4-3-1)}, we have a contraction $\widetilde{\pi}:\widetilde{V} \to \widetilde{X}$ of $\widetilde{D}$, so that $(\widetilde{X} ,\widetilde{\pi}_{\ast}(\widetilde{C}))$ is a minimal compactification of $\bC ^2$ with $K_{\widetilde{X}} \equiv 0$. 
Let $\widetilde{D}=\sum _j\widetilde{D}_j$ be the decomposition of $\widetilde{D}$ into irreducible components. 
Since the intersection matrix of $\widetilde{D}$ is negative definite, there exists uniquely an effective $\bQ$-divisor $\widetilde{D}^{\natural}=\sum _j \widetilde{\alpha}_j\widetilde{D}_j$ on $\widetilde{V}$ such that $(\widetilde{D}_j \cdot K_{\widetilde{V}}+\widetilde{D}^{\natural})=0$ for every irreducible component $\widetilde{D}_j$ of $\widetilde{D}$. 
Then we know that $(\widetilde{C} \cdot \widetilde{D}^{\natural}) = 1$ by Lemma \ref{lem(3-2-2)} (2). 
By Lemma \ref{lem(3-2-4)}, we know $(C \cdot D^{\natural}) < 1$, where notice $(C \cdot D^{\natural}) \not= 1$ by Lemma \ref{lem(4-3-1)}. 
Hence, $-K_X$ is numerically ample by Lemma \ref{lem(3-2-2)} (1). 
\medskip

\noindent
{\bf Case 2:} $d(A)-d(\overline{A}) < \ell \le d(A)(n\,d(A) - d(\overline{A})) - 2$. 
In this case, there exists the sequence $\widetilde{f} :V \to \widetilde{V}$ of contractions of (smoothly) contractible components in $\Supp (C+D)$ such that the weighted dual graph of $\widetilde{C} + \widetilde{D} := \widetilde{f}^{-1}_{\ast}(C+D)_{\rm red.}$ looks like that in Figure \ref{fig(1)}, where $\widetilde{C}$ corresponds to the vertex with label $C$ in Figure \ref{fig(1)}. 
Let $\widetilde{D}=\sum _j\widetilde{D}_j$ be the decomposition of $\widetilde{D}$ into irreducible components. 
By a similar argument to Case 1, the intersection matrix of $\widetilde{D}$ is negative definite, furthermore; there exists uniquely an effective $\bQ$-divisor $\widetilde{D}^{\natural}=\sum _j \widetilde{\alpha}_j\widetilde{D}_j$ on $\widetilde{V}$ such that $(\widetilde{D}_j \cdot K_{\widetilde{V}}+\widetilde{D}^{\natural})=0$ for every irreducible component $\widetilde{D}_j$ of $\widetilde{D}$ and $(\widetilde{C} \cdot \widetilde{D}^{\natural}) = 1$. 
By Lemma \ref{lem(3-2-4)}, we know $(C \cdot D^{\natural}) > 1$, where notice $(C \cdot D^{\natural}) \not= 1$ by Lemma \ref{lem(4-3-1)}. 
Hence, $K_X$ is numerically ample by Lemma \ref{lem(3-2-2)} (3). 
\medskip

The proof of Theorem \ref{main(3)} (2) is thus completed. 
\end{proof}
%%%%%%%%%%%%%%%%%
\begin{proof}[Proof of Theorem \ref{main(3)} (3)]
By assumption, note that the weighted dual graph of $C+D$ is given as (4) in Appendix \ref{5}. 
Hence, there exists a sequence $\widetilde{f} :V \to \widetilde{V}$ of contractions of (smoothly) contractible components in $\Supp (C+D)$ such that the weighted dual graph of $\widetilde{C} + \widetilde{D} := \widetilde{f}_{\ast}(D)$ is the following: 
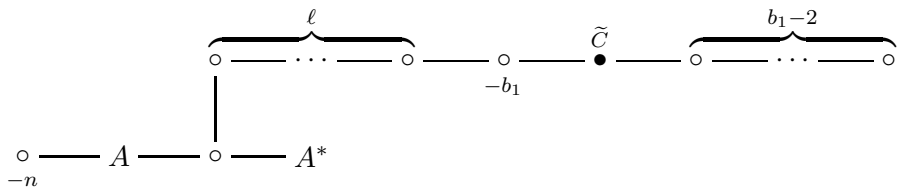
\begin{align*}
\xygraph{
{A^{\ast}} -[l] \circ (-[u] \circ -[r] \cdots ([]!{+(0,.35)} {\overbrace{\quad \qquad \qquad \qquad}^{\ell}}) -[r] \circ -[r] \circ ([]!{+(0,-.3)} {^{-b_1}}) -[r] \bullet ([]!{+(0,.2)} {^{\widetilde{C}}}) -[r] \circ -[r] \cdots ([]!{+(0,.35)} {\overbrace{\quad \qquad \qquad \qquad}^{b_1-2}}) -[r] \circ ,-[l] {A} -[l] \circ ([]!{+(0,-.3)} {^{-n}})}, 
\end{align*}
Notice that the above weighted dual graph is a special situation that (4) in Appendix \ref{5}. 
Hence, the intersection matrix of $\widetilde{D}$ is negative definite by Theorem \ref{main(2)} and Lemma \ref{cont}. 
Let $\widetilde{D}=\sum _j\widetilde{D}_j$ be the decomposition of $\widetilde{D}$ into irreducible components. 
Then there exists uniquely an effective $\bQ$-divisor $\widetilde{D}^{\natural}=\sum _j \widetilde{\alpha}_j\widetilde{D}_j$ on $\widetilde{V}$ such that $(\widetilde{D}_j \cdot K_{\widetilde{V}}+\widetilde{D}^{\natural})=0$ for every irreducible component $\widetilde{D}_j$ of $\widetilde{D}$. 
Now, we shall focus on the non-negative integer $\ell$. 
We consider the following three cases separately. 
\medskip

\noindent
{\bf Case 1:} $0 \le \ell < (n+1)\,d(A)-d(\overline{A}) - 1$. 
In this case, we obtain $(\widetilde{C} \cdot \widetilde{D}^{\natural}) < 1$ by Theorem \ref{main(3)} (2) and Lemma \ref{lem(3-2-2)} (1). 
By Lemma \ref{lem(3-2-4)}, we know $(C \cdot D^{\natural}) < 1$, where notice $(C \cdot D^{\natural}) \not= 1$ by $\widetilde{f} \not= id$ and Lemma \ref{lem(4-3-1)}. 
Hence, $-K_X$ is numerically ample by Lemma \ref{lem(3-2-2)} (1). 
\medskip

\noindent
{\bf Case 2:} $\ell = (n+1)\,d(A)-d(\overline{A}) - 1$. 
In this case, we consider the following two subcases separately. 
\medskip

\noindent
{\bf Subcase 2-1:} $s=1$. 
In this subcase, we know that $\widetilde{f} = id$. 
Hence, $K_X$ is numerically trivial by Lemma \ref{lem(4-3-1)}. 
\medskip

\noindent
{\bf Subcase 2-2:} $s>1$. 
In this subcase, we obtain $(\widetilde{C} \cdot \widetilde{D}^{\natural})=1$ by Theorem \ref{main(3)} (2) and Lemma \ref{lem(3-2-2)} (2).  
By Lemma \ref{lem(3-2-4)}, we know $(C \cdot D^{\natural}) > 1$, where notice $(C \cdot D^{\natural}) \not= 1$ by $\widetilde{f} \not= id$ and Lemma \ref{lem(4-3-1)}. 
Hence, $K_X$ is numerically ample by Lemma \ref{lem(3-2-2)} (3). 
\medskip

\noindent
{\bf Case 3:} $(n+1)\,d(A)-d(\overline{A}) - 1 < \ell \le d(A)(n\,d(A) - d(\overline{A})) - 2$. 
In this case, we obtain $(\widetilde{C} \cdot \widetilde{D}^{\natural}) > 1$ by Theorem \ref{main(3)} (2) and Lemma \ref{lem(3-2-2)} (3). 
By Lemma \ref{lem(3-2-4)}, we know $(C \cdot D^{\natural}) > 1$. 
Hence, $K_X$ is numerically ample by Lemma \ref{lem(3-2-2)} (3). 
\medskip

The proof of Theorem \ref{main(3)} (3) is thus completed. 
\end{proof}
%%%%%%%%%%%%%%%%%
\begin{proof}[Proof of Theorem \ref{main(3)} (4)]
By assumption, note that the weighted dual graph of $C+D$ is given as (5) in Appendix \ref{5}. 
Hence, there exists a sequence $\widetilde{f} :V \to \widetilde{V}$ of contractions of (smoothly) contractible components in $\Supp (C+D)$ such that the weighted dual graph of $\widetilde{C} + \widetilde{D} := \widetilde{f}_{\ast}(D)$ is the following: 
\begin{align*}
\xygraph{
{A^{\ast}} -[l] \circ (-[u] \circ -[r] \cdots ([]!{+(0,.35)} {\overbrace{\quad \qquad \qquad \qquad}^{\ell}}) -[r] \circ -[r] \circ ([]!{+(0,-.3)} {^{-b_1}}) -[r] \cdots -[r] \circ ([]!{+(0,-.3)} {^{-b_s}}) -[r] \bullet ([]!{+(0,.2)} {^{\widetilde{C}}}) -[r] {\underline{B^{\ast}}},-[l] {A} -[l] \circ ([]!{+(0,-.3)} {^{-n}})}
\end{align*}
Notice that the above weighted dual graph is a special situation that (4) in Appendix \ref{5}. 
Hence, the intersection matrix of $\widetilde{D}$ is negative definite by Theorem \ref{main(2)} and Lemma \ref{cont}. 
Let $\widetilde{D}=\sum _j\widetilde{D}_j$ be the decomposition of $\widetilde{D}$ into irreducible components. 
Then there exists uniquely an effective $\bQ$-divisor $\widetilde{D}^{\natural}=\sum _j \widetilde{\alpha}_j\widetilde{D}_j$ on $\widetilde{V}$ such that $(\widetilde{D}_j \cdot K_{\widetilde{V}}+\widetilde{D}^{\natural})=0$ for every irreducible component $\widetilde{D}_j$ of $\widetilde{D}$. 
Now, we shall focus on the non-negative integer $\ell$. 
We consider the following two cases separately. 
\medskip

\noindent
{\bf Case 1:} $0 \le \ell < (n+1)\,d(A)-d(\overline{A}) - 1$. 
In this case, we obtain $(\widetilde{C} \cdot \widetilde{D}^{\natural}) < 1$ by Theorem \ref{main(3)} (3) and Lemma \ref{lem(3-2-2)} (1). 
By Lemma \ref{lem(3-2-4)}, we know $(C \cdot D^{\natural}) < 1$, where notice $(C \cdot D^{\natural}) \not= 1$ by Lemma \ref{lem(4-3-1)}. 
Hence, $-K_X$ is numerically ample by Lemma \ref{lem(3-2-2)} (1). 
\medskip

\noindent
{\bf Case 2:} $(n+1)\,d(A)-d(\overline{A}) - 1 \le \ell  \le d(A)(n\,d(A) - d(\overline{A})) - 2$. 
In this case, we obtain $(\widetilde{C} \cdot \widetilde{D}^{\natural}) \ge 1$ by Theorem \ref{main(3)} (3) and Lemma \ref{lem(3-2-2)} (2) and (3). 
By Lemma \ref{lem(3-2-4)}, we know $(C \cdot D^{\natural}) > 1$, where notice $(C \cdot D^{\natural}) \not= 1$ by Lemma \ref{lem(4-3-1)}. 
Hence, $K_X$ is numerically ample by Lemma \ref{lem(3-2-2)} (3). 
\medskip

The proof of Theorem \ref{main(3)} (4) is thus completed. 
\end{proof}
%%%%%%%%%%%%%%%%%
\begin{proof}[Proof of Theorem \ref{main(3)} (1)]
If the weighted dual graph of $C+D$ is given as one of (1) and (2) in Appendix \ref{5}, then $X$ has at most quotient singularities; hence, $-K_X$ is numerically ample by {\cite[Theorem 1.1 (1)]{KT09}}. 

In what follows, we assume that the weighted dual graph of $C+D$ is given as one of (6) and (7) in Appendix \ref{5}. 
Then there exists a sequence $\widetilde{f} :V \to \widetilde{V}$ of contractions of (smoothly) contractible components in $\Supp (C+D)$ such that the weighted dual graph of $\widetilde{C} + \widetilde{D} := \widetilde{f}_{\ast}(D)$ is as follows: 
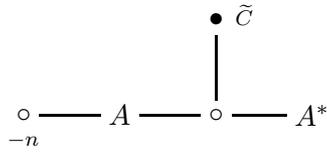
\begin{align*}
\xygraph{{A^{\ast}} -[l] \circ (-[u] \bullet ([]!{+(.3,0)} {^{\widetilde{C}}}), -[l] {A} -[l] \circ ([]!{+(0,-.3)} {^{-n}}))}
\end{align*}
Since $\widetilde{D}$ is a rational chain, we have the contraction $\widetilde{\pi}:\widetilde{V} \to \widetilde{X}$ of $\widetilde{D}$. 
Moreover, $(\widetilde{X},\widetilde{\pi}_{\ast}(\widetilde{C}))$ is a minimal compactification of $\bC ^2$ with exactly one cyclic quotient singular point. 
Hence, $-K_{\widetilde{X}}$ is numerically ample by {\cite[Theorem 1.1 (1)]{KT09}}. 
Let $\widetilde{D}=\sum _j\widetilde{D}_j$ be the decomposition of $\widetilde{D}$ into irreducible components. 
Since the intersection matrix of $\widetilde{D}$ is negative definite, there exists uniquely an effective $\bQ$-divisor $\widetilde{D}^{\natural}=\sum _j \widetilde{\alpha}_j\widetilde{D}_j$ on $\widetilde{V}$ such that $(\widetilde{D}_j \cdot K_{\widetilde{V}}+\widetilde{D}^{\natural})=0$ for every irreducible component $\widetilde{D}_j$ of $\widetilde{D}$. 
By Lemma \ref{lem(3-2-2)} (1), we know that $(\widetilde{C} \cdot \widetilde{D}^{\natural}) < 1$. 

Suppose that $-K_X$ is not numerically ample. 
Then $K_X$ must be numerically ample by Lemma \ref{lem(4-3-1)}. 
Hence, $(C \cdot D^{\natural}) > 1$ by Lemma \ref{lem(3-2-2)} (3). 
By Lemma \ref{lem(3-2-4)}, $\widetilde{f}$ can be factorized into $\mu :V \to \hat{V}$ and $\nu :\hat{V} \to \widetilde{V}$; in addition, there exist a unique $(-1)$-curve $\hat{C}$ on $\mu _{\ast}(D)$ and a contraction $\hat{\pi}: \hat{V} \to \hat{X}$ of $\hat{D} := \mu _{\ast}(D)-\hat{C}$ such that $(\hat{X},\hat{\pi}(\hat{C}))$ is a minimal compactification of $\bC ^2$ with the numerically trivial canonical divisor $K_{\hat{X}}$. 
Since $\mu \not= \widetilde{f}$, we notice that $\hat{D}$ has a branching component with self-intersection number $<-2$. 
However, this is a contradiction to Lemma \ref{lem(4-3-1)}. 
This proves Theorem \ref{main(3)} (1). 
\end{proof}
%%%%%%%%%%%%%%%%%
The proof of Theorem \ref{main(3)} is thus completed. 
%%%%%%%%%%%%%%%%%%%%%%%%%%%%%%%%%%%%%%%%%%%%%%%%%%%%%%%%%%%%%%%%%%%%%%%%%%%%%%%%%%%%%%%%%%%%%%%%%%%%%%%%%%
\appendix
\section{List of configurations}\label{5}
%%%%%%%%%%%%%%%%%
In this appendix, we summarize configurations of all weighted dual graphs of $C+D$, where $C$ and $D$ are the same as in Theorem \ref{main}. 
We employ the following notation:
\begin{itemize}
\item In (1)--(7), $n \ge 2$; 
\item In (2)--(7), $A$ is an admissible twig and $A^{\ast}$ is the adjoint of $A$; 
\item In (3), (4) and (5), $0 \le \ell \le d(A)(n\,d(A)-d(\overline{A}))-2$; 
\item In (4)--(7), $[b_1,\dots ,b_s]$ is an admissible twig with $b_1 \ge 3$ and $\underline{B^{\ast}}$ is the adjoint of $[b_1,\dots ,b_s]$ removed the last component; 
\item In (5) and (7), $m \ge 0$. 
\end{itemize}
%%%%%%%%%%%%%%%%%
\begin{tikzpicture}
\node at (0,1.5) {(1)};
\node at (0.5,0) {\xygraph{\circ ([]!{+(0,.2)} {^{C}}) ([]!{+(0,-.3)} {^{0}}) -[l] \circ ([]!{+(0,-.3)} {^{-n}})}};

\node at (3,1.5) {(2)};
\node at (5,0) {\xygraph{{A^{\ast}} -[l] \bullet ([]!{+(0,.2)} {^{C}}) -[l] {A} -[l] \circ ([]!{+(0,-.3)} {^{-n}})}};

\node at (9,1.5) {(3)};
\node at (12,0.75) {\xygraph{
{A^{\ast}} -[l] \circ (-[u] \circ -[r] \cdots ([]!{+(0,.35)} {\overbrace{\quad \qquad \qquad \qquad}^{\ell}}) -[r] \circ -[r] \bullet ([]!{+(0,.2)} {^{C}}) ,-[l] {A} -[l] \circ ([]!{+(0,-.3)} {^{-n}})}};
\end{tikzpicture}

\begin{tikzpicture}
\node at (0,-1.5) {(4)};
\node at (5.5,-2.25) {\xygraph{
{A^{\ast}} -[l] \circ (-[u] \circ -[r] \cdots ([]!{+(0,.35)} {\overbrace{\quad \qquad \qquad \qquad}^{\ell}}) -[r] \circ -[r] \circ ([]!{+(0,-.3)} {^{-b_1}}) -[r] \cdots -[r] \circ ([]!{+(0,-.3)} {^{-b_s}}) -[r] \bullet ([]!{+(0,.2)} {^{C}}) -[r] {\underline{B^{\ast}}},-[l] {A} -[l] \circ ([]!{+(0,-.3)} {^{-n}})}};
\end{tikzpicture}

\begin{tikzpicture}
\node at (0,-4.5) {(5)};
\node at (7,-5.5) {\xygraph{
{A^{\ast}} -[l] \circ (-[u] \circ -[r] \cdots ([]!{+(0,.35)} {\overbrace{\quad \qquad \qquad \qquad}^{\ell}}) -[r] \circ -[r] \circ ([]!{+(0,-.3)} {^{-b_1}}) -[r] \cdots -[r] \circ ([]!{+(0,-.3)} {^{-b_s}}) -[r] \circ ([]!{+(0,-.3)} {^{-(m+2)}}) (-[r] {\underline{B^{\ast}}},-[u] \bullet ([]!{+(0,.2)} {^{C}}) -[r] \circ -[r] \cdots ([]!{+(0,.35)} {\overbrace{\quad \qquad \qquad \qquad}^{m}}) -[r] \circ ),-[l] {A} -[l] \circ ([]!{+(0,-.3)} {^{-n}})}};
\end{tikzpicture}

\begin{tikzpicture}
\node at (0,-8.5) {(6)};
\node at (4,-9.25) {\xygraph{
{A^{\ast}} -[l] \circ ([]!{+(0,-.3)} {^{-b_1}}) (-[u] \circ ([]!{+(.3,-.3)} {^{-b_2}}) -[r] \cdots -[r] \circ ([]!{+(0,-.3)} {^{-b_s}}) -[r] \bullet ([]!{+(0,.2)} {^{C}}) -[r] {\underline{B^{\ast}}},-[l] {A} -[l] \circ ([]!{+(0,-.3)} {^{-n}})}};
\end{tikzpicture}

\begin{tikzpicture}
\node at (0,-11.5) {(7)};
\node at (5,-12.5) {\xygraph{
{A^{\ast}} -[l] \circ ([]!{+(0,-.3)} {^{-b_1}}) (-[u] \circ ([]!{+(.3,-.3)} {^{-b_2}}) -[r] \cdots -[r] \circ ([]!{+(0,-.3)} {^{-b_s}}) -[r] \circ ([]!{+(0,-.3)} {^{-(m+2)}}) (-[r] {\underline{B^{\ast}}},-[u] \bullet ([]!{+(0,.2)} {^{C}}) -[r] \circ -[r] \cdots ([]!{+(0,.35)} {\overbrace{\quad \qquad \qquad \qquad}^{m}}) -[r] \circ ),-[l] {A} -[l] \circ ([]!{+(0,-.3)} {^{-n}})}};
\end{tikzpicture}
%%%%%%%%%%%%%%%%%%%%%%%%%%%%%%%%%%%%%%%%%%%%%%%%%%%%%%%%%%%%%%%%%%%%%%%%%%%%%%%%%%%%%%%%%%%%%%%%%%%%%%%%%%

\end{document}